\documentclass[12pt]{article}

\usepackage{amsfonts} 
\usepackage{enumerate}
\usepackage{graphicx}
\usepackage{fancyhdr, color}
\usepackage{paralist}
\usepackage{amsmath}
\usepackage[utf8]{inputenc}
\usepackage[english]{babel}
\usepackage{amsthm}
\usepackage{mathrsfs}
\usepackage{enumitem}
\usepackage{tikz}
\usepackage{chngcntr}
\usepackage{authblk}

\newtheorem{theorem}{Theorem}[section]
\newtheorem{proposition}{Proposition}[section]

\newtheorem{property}{Property}[section]
\newtheorem{definition}{Definition}[section]
\newtheorem{remark}{Remark}[section]
\newtheorem{example}{Example}[section]

\counterwithout{figure}{section}
\counterwithout{figure}{subsection}

\begin{document}

\title{Probabilistic Properties of GIG Digraphs}

\author[1]{Chuhan Guo \thanks{Research supported by Davidson Research Initiative}}
\author[1]{Laurie J. Heyer \thanks{Research supported by Davidson Research Initiative and NSF grant MCB-1613203 to Davidson College}} 
\affil[1]{ \normalsize Department of Mathematics and Computer Science\authorcr Davidson College} 
\author[2]{Jeffrey L. Poet \thanks{Research supported by NSF grant MCB-1613281 }}
\affil[2]{\normalsize Department of Computer Science, Mathematics and Physics\authorcr Missouri Western State University}

\maketitle

\begin{abstract}
We study the probabilistic properties of the Greatest Increase Grid (GIG) digraph. We compute the probability of a particular sequence of directed edges connecting two random vertices. We compute the joint probability that a set of vertices are all sinks, and derive the mean and variance in the number of sinks in a randomly labeled GIG digraph. Finally, we show that the expected size of the maximum component of vertices converges.  
\end{abstract}

\section{Introduction}

Local search is a heuristic approach to solve large-scale and computationally challenging global optimization problems. A simple but fundamental local search algorithm, the hill-climbing algorithm searches for the global maximum of a function $L(x,y)$ by making the optimal choice based on the gradient at each $(x,y)$ \cite{artificial intelligence}\relax. We consider a discrete version of the hill-climbing algorithm by restricting the search to an $m \times n$ integer lattice, and the neighborhood of each point $(i,j)$ in the lattice to be the two horizontal and two vertical neighbors. We further assume that each function value on the lattice is unique, so the values can be ordered from $1$ to $mn$. The steepest ascent hill-climbing algorithm determines the direction of the greatest increase and moves to the adjacent lattice point in that direction. This algorithm terminates when reaching a local maximum.

We define a graph theoretic representation of the algorithm by letting the vertex set $V$ be an $m \times n$ integer lattice and the directed edge set $E$ represent the direction of the greatest increase from each vertex. Specifically, for each vertex $V_i \in V$, let $v_i \in \{1,2,\ldots, mn \}$ denote the label of vertex $V_i$.   
For $W \subset V$, denote by $N(W)$ the set that contains vertices in set $W$ and neighbors of all vertices in $W$, where we define neighbors as vertices at unit Euclidean distance away, i.e., one unit to the north, south, east and west in the lattice. (Note that throughout this paper, we use {\em neighbor} to refer to proximity in the lattice, rather than adjacency in the digraph.) 

Let $n(W)$ be the set of all labels of vertices in $N(W)$. We assume that all $mn$ labels are distinct.
Then the directed edge $(V_i, V_j) \in E$ if and only if $v_j = \max (n(V_i))$ and $i \neq j$.  The directed graph $G = \{V, E \}$ with these restrictions is a Greatest Increase Grid (GIG) Digraph, as introduced by Chester, et al.\ \cite{enum paper}\relax and further characterized by Allen et al.\ \cite{main}\relax. 

We employ coordinate notation for vertices when needed to conveniently refer to a specific vertex in the lattice. In this notation, for integers $1 \leq i \leq m$ and $1 \leq j \leq n$, $V_{i,j} \in V$ denotes the vertex  located at row $i$ and column $j$ of the $m \times n$ lattice. The coordinate notation does not conflict with the single subscript notation; we let the context determine the notation. If outdegree$(V_i) = 0$, then vertex $V_i$ is called a sink. Figure \ref{fig:GIG} shows an example of a $3 \times 3$ GIG digraph. The vertices with labels $6$, $8$, and $9$  ($V_{3,1}$, $V_{3,3}$, and $V_{1,2}$, respectively, in coordinate notation) are sinks.

\begin{figure}[ht]
\centering
\begin{tikzpicture}
[scale=0.8,every node/.style={draw=black,circle}]

    \node (6) at (0,0) {6};
    \node (1) at (1.5,0) {1};
    \node (8) at (3,0) {8};
    \node (4) at (0,1.5) {4};
    \node (7) at (1.5,1.5) {7};
    \node (3) at (3,1.5) {3};
    \node (2) at (0,3) {2};
    \node (9) at (1.5,3) {9};
    \node (5) at (3,3) {5};
    
    \draw[->] (2) to (9);
    \draw[->] (5) to (9);
    \draw[->] (4) to (7);
    \draw[->] (7) to (9);
    \draw[->] (3) to (8);
    \draw[->] (1) to (8);
   
\end{tikzpicture}
\caption{A $3 \times 3$ GIG digraph} \label{fig:GIG}
\end{figure}
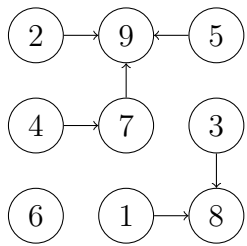

GIG digraphs are special cases of Limited Outdegree Grid (LOG) digraphs \cite{main, enum paper}. A LOG digraph is defined on a lattice like a GIG. Each vertex of a LOG digraph has outdegree at most one. A GIG digraph is a LOG digraph with labels $\{1,2,\ldots mn \}$ and corresponding restrictions on edges based on the labels of vertices. Allen et al.\ introduce algorithms for recognizing GIGs and LOGs and discuss an application to map folding problems \cite{main}\relax. Chester et al.\ focus on properties of subgraphs and enumerate all possible $3 \times 3$ LOG and GIG digraphs \cite{enum paper}\relax. This paper will explore the probabilistic properties of randomly labeled GIG digraphs.

Three properties of the GIG digraph are studied in this paper. In section $2$, we compute the probability of a path connecting two random vertices. In section $3$, we determine some probabilities and statistical properties of sinks. Section $4$ investigates the expected size of a component of vertices. These probabilistic properties may improve the decision making in randomized and adaptive perturbations of iterated local searches like stochastic gradient descent \cite{lourenco}\relax.

\section{Probability of connectedness}

In this section, we compute the probability that a randomly labeled $m \times n$ grid will produce a GIG digraph with a particular directed path. We present two different proofs, one enumerative and one direct probability proof. To simplify notation, we denote $ |N(V_1, \ldots, V_j)|$ by $K_j$.

%
%
%
%

\begin{theorem}\label{N step probability}
Let $G = \{V, E\}$ be an $m \times n$ GIG digraph containing the directed path $(V_1, \ldots, V_i )$. The probability that a randomly labeled GIG digraph contains the directed path $(V_1, \ldots, V_i)$ is $\frac{1}{K_1} \cdot \frac{1}{K_2} \cdot \cdots \cdot \frac{1}{K_{i-1}}$. 
\end{theorem}

\begin{example} 
Consider the directed path $(V_{2,1}, V_{2,2}, V_{1,2})$ of length 2 in Figure \ref{fig:GIG}. We have the following:

 $N(V_{2,1}) = \{V_{2,1}, V_{1,1} V_{2,2}, V_{3,1}\}$
 
 $K_1= 4$  
 
 $N(V_{2,1}, V_{2,2}) = \{V_{2,1}, V_{2,2}, V_{1,1}, V_{3,1}, V_{1,2}, V_{2,3},V_{3,2}\}$ 
 
 $K_2 = 7$
 
 $P( ( V_{2,1}, V_{2,2}, V_{1,2} ) ) = \frac{1}{4} \cdot \frac{1}{7} = \frac{1}{28}$
\end{example}

\subsection{An enumerative proof of Theorem \ref{N step probability}}

\begin{proof}
Consider the set $N(V_1, \ldots, V_{i-1})$ that contains the first $i-1$ vertices on this directed path and the neighbors of these vertices. There are $K_{i-1}$ vertices in this set with $K_{i-1}$ distinct labels. The formation of the directed path is only contingent on the relative relationship among these $K_{i-1}$ labels. The size of the GIG digraph and the specific values of these labels do not matter. 

Note that vertex $V_i$ is a neighbor of vertex $V_{i-1}$, so $V_i$ is also in $N(V_1, \ldots, V_{i-1})$. Since there is an edge pointing from $V_{i-1}$ to $V_{i}$, $v_i > v_{i-1}$ and $v_i$ is the largest label in $n(V_{i-1})$.  
Similarly, $v_{i-1} >  v_{i-2}$ and $v_{i-1}$ is the largest label in $n(V_{i-2})$. Inductively, we know that the label of $V_i$ is the greatest in $n(V_1, \ldots, V_{i-1})$. 

Other than the largest value in $n(V_1, \ldots, V_{i-1})$, any  values are legitimate ones to label the $K_{i-2}$ vertices in $N(V_1, \ldots, V_{i-2})$. Thus, there are $\binom{K_{i-1}-1}{K_{i-2}}$ ways to choose labels for vertices in set $N(V_1, \ldots, V_{i-2})$. Fixing the labels of these $K_2$ vertices and the label of vertex $V_i$, the other neighbors of $V_{i-1}$ that are not in the set $N(V_1, \ldots, V_{i-2})$ can be labeled in $(K_{i-1}- K_{i-2} - 1)!$ ways.  

In general, for $1 \leq l \leq i-1$, by applying the argument provided above inductively to the set $N(V_1, \ldots, V_{i-l})$, we will have the following analytical formula to enumerate the number of possible labelings that generates the particular path $(V_1, \ldots, V_i)$:
\begin{equation*}
\binom{K_{i-1}-1}{K_{i-2}} \binom{K_{i-2}-1}{K_{i-3}} \cdots \binom{K_{2}-1}{K_{1}} (K_{i-1}- K_{i-2} - 1)! (K_{i-2}- K_{i-3} - 1)! \cdots (K_{1} - 1)!
\end{equation*}

The number of ways to put the $K_{i-1}$ labels on these $K_{i-1}$ vertices is $K_{i-1}!$. The probability that a randomly labeled GIG digraph contains the directed path $(V_1, \ldots, V_i)$ is thus the formula above over  $K_{i-1}!$, which can be easily simplified into $\frac{1}{K_1} \cdot \frac{1}{K_2} \cdot \cdots \cdot \frac{1}{K_{i-1}}$.
\end{proof}

\subsection{A direct probability proof of theorem \ref{N step probability}}
%
%
%
%

\begin{proof}
If there exists a directed path from vertex $V_1$ to vertex $V_2$, then the label $v_2$ is greater than $v_1$ and the labels of other neighbors of $V_1$. The probability that $v_2$ has the greatest value among these labels is $\frac{1}{K_1}$.

Likewise, if $(V_1, V_2, V_3)$ is a directed path in a GIG digraph, then $v_3$ is greater than  $v_2$ and labels of all other neighbors of $v_2$. Similarly, $v_2$ needs to be greater than all labels in $n(V_1) $. Therefore $v_3$ is greater than all labels in $n(\{V_1, V_2\})$.
\begin{align*}
P(v_1 \to v_2 \to v_3) &= P(v_3 =\max( n(\{V_1, V_2 \})) , v_2 = \max( n(V_1))) \\
&= P(v_3 =\max( n(\{V_1,V_2\}))) P( v_2 = \max( n(V_1))) \\
&= \frac{1}{K_2} \cdot \frac{1}{K_1}, 
\end{align*}
where the second line follows because $v_3$ being the largest label in $n(\{V_1, V_2\})$ is independent of $v_2$ being the largest label in $n(V_1)$.

Inductively, the probability that a randomly labeled GIG digraph contains the directed path $(V_1, \ldots, V_i)$ is $\frac{1}{K_1} \cdot  \frac{1}{K_2} \cdot\cdots \cdot \frac{1}{K_{i-1}}$
\end{proof}

\begin{remark}\normalfont
Since the probability of a random path only depends on the number of new neighbors at each step, Theorem \ref{N step probability} can be applied to GIG digraph of any shape and any dimension. 
\end{remark}

\subsection{Connectedness of two vertices}

As a consequence of Theorem \ref{N step probability}, it is easy to see that the probability of a particular path in a randomly labeled GIG digraph is dependent upon its length and perhaps other factors that influence the number of neighbors of the vertices of the path. In this section we demonstrate three properties, one at a time, that affect the probability of a path while leaving the other factors unchanged.  

Figure \ref{fig:5by5GIG} shows an example of an unlabeled $5 \times 5$ GIG digraph with three potential paths from $V_{4,1}$ to $V_{1,4}$. Denote path $(V_{4,1}, V_{4,2}, V_{4,3}, V_{4,4}, V_{3,4}, V_{2,4}, V_{1,4})$ by $P_1$, path $(V_{4,1}, V_{4,2}, V_{4,3}, V_{3,3},$ $V_{3,4}, V_{2,4}, V_{1,4})$ by $P_2$, and path $(V_{4,1}, V_{4,2}, V_{3,2}, V_{3,3}, V_{3,4}, V_{2,4}, V_{1,4})$ by $P_3$. Note that no two of these three paths can exist simultaneously in a labeled GIG digraph.

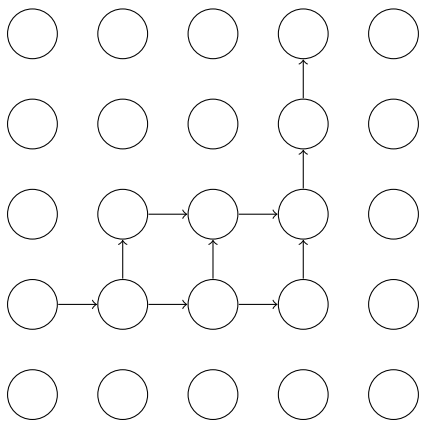
\begin{figure}[ht]
\centering
\begin{tikzpicture}
[scale=.8,every node/.style={draw=black,circle}]

    \node[scale = 1.7] (1) at (0,0) {};
    \node[scale = 1.7] (2) at (1.5,0) {};
    \node[scale = 1.7] (3) at (3,0) {};
    \node[scale = 1.7] (4) at (4.5,0) {};
    \node[scale = 1.7] (13) at (6,0) {};
    \node[scale = 1.7] (5) at (0,1.5) {};
    \node[scale = 1.7] (6) at (1.5,1.5) {};
    \node[scale = 1.7] (7) at (3,1.5) {};
    \node[scale = 1.7] (8) at (4.5,1.5) {};
    \node[scale = 1.7] (14) at (6,1.5) {};
    \node[scale = 1.7] (9) at (0,3) {};
    \node[scale = 1.7] (10) at (1.5,3) {};
    \node[scale = 1.7] (11) at (3,3) {};
    \node[scale = 1.7] (12) at (4.5,3) {};
    \node[scale = 1.7] (15) at (6,3) {};
    \node[scale = 1.7] (16) at (0,-1.5) {};
    \node[scale = 1.7] (17) at (1.5,-1.5) {};
    \node[scale = 1.7] (18) at (3,-1.5) {};
    \node[scale = 1.7] (19) at (4.5,-1.5) {};
    \node[scale = 1.7] (20) at (6,-1.5) {};
    \node[scale = 1.7] (21) at (0,4.5) {};
    \node[scale = 1.7] (22) at (1.5,4.5) {};
    \node[scale = 1.7] (23) at (3,4.5) {};
    \node[scale = 1.7] (24) at (4.5,4.5) {};
    \node[scale = 1.7] (25) at (6,4.5) {};
    
    \draw[->] (1) to (2);
    \draw[->] (2) to (3);
    \draw[->] (3) to (4);
    \draw[->] (2) to (6);
    \draw[->] (3) to (7);
    \draw[->] (4) to (8);
    \draw[->] (6) to (7);
    \draw[->] (7) to (8);
    \draw[->] (8) to (12);
    \draw[->] (12) to (24);
   
\end{tikzpicture}
\caption{Three potential paths in an unlabeled $5 \times 5$ GIG digraph} \label{fig:5by5GIG}
\end{figure}

%
%
%
%

\begin{definition}
In a directed path $(V_1, V_2, \ldots, V_i)$, if edge $(V_{j-1}, V_{j})$ is perpendicular to edge $(V_{j}, V_{j+1})$, then vertex $V_j$ is a turn. 
\end{definition}

\begin{example}
Note that in path $P_1$, edge $(V_{4,3}, V_{4,4})$ is perpendicular to edge $(V_{4,4}, V_{3,4})$. Vertex $V_{4,4}$ is the only turn in $P_1$. 
\end{example}

%
%
\begin{property}\label{more turns higher prob}
A directed path with a particular number of turns is more likely to occur in a randomly labeled GIG digraph than a directed path of the same length with fewer turns.
\end{property}

Let $V_1, V_2, \ldots, V_i$ be a directed path.
For any $2 \leq p \leq i-1$, if $V_p$ is a turn in a path, 
then $V_{p-1}$ and $V_{p+1}$ will share one more neighbor than the case where $V_p$ is not a turn, which implies that the number of new neighbors $(K_{p+1} - K_p)$ is one fewer if $V_p$ is a turn. Therefore, we will witness a decrease of one in the values of $K_{p+1}, K_{p+2}, \ldots, K_{i-1}$, increasing the probability that this directed path exists. 

To illustrate Property \ref{more turns higher prob}, the path $P_2$ in Figure \ref{fig:5by5GIG} has more turns than path $P_1$. Among all the randomly labeled GIG digraphs, the probability that the path $P_1$ exists is $\frac{1}{4} \cdot \frac{1}{7} \cdot \frac{1}{10} \cdot \frac{1}{13} \cdot \frac{1}{15} \cdot \frac{1}{18}$. The probability that the path $P_2$ exists is  $\frac{1}{4} \cdot \frac{1}{7} \cdot \frac{1}{10} \cdot \frac{1}{12} \cdot \frac{1}{14} \cdot \frac{1}{16}$, which is greater than that of path $P_1$.

%
%
%
%
\begin{property}\label{earlier turns}
A directed path with particular locations of turns is more likely to occur in a randomly labeled GIG digraph than a directed path with one or more turns occurring in later locations in the directed path, but all else identical.
\end{property}

Let $n_1, n_2 \in \mathbb{N}$ be such that $n_1 < n_2$. 
Denote two factors in the analytical formula for the probability of the existence of a directed path in a randomly labeled GIG digraph by $\frac{1}{n_1}$ and $\frac{1}{n_2}$. Note that if there is a turn at vertex $V_t$, then the value of the corresponding $K_{t+1}$ will be $1$ less than the case where $V_t$ is not a turn. We will use $n_1 - 1$ and $n_2-1$ to capture the influence of a turn in an earlier part and a later part of a directed path on the probability of existence of such a path. Note that
\begin{align*}
\frac{1}{n_1-1} \cdot \frac{1}{n_2} - \frac{1}{n_1} \cdot \frac{1}{n_2-1} = \frac{n_2 -n_1}{n_1n_2(n_1-1)(n_2-1)} > 0,
\end{align*}  
which implies that an earlier turn in the path is associated with a higher probability of existence for this path.

For instance, although path $P_2$ and path $P_3$ each have two turns, the first turn of path $P_3$ takes place at the second vertex of the path, while the first turn of $P_2$ takes place at the third vertex. The second turn of both paths locates at the fifth vertex. Thus, by Property \ref{earlier turns}, path $P_3$ is more likely to occur in a randomly labeled GIG digraph than path $P_2$. Numbers show that among all the randomly labeled GIG digraphs, the probability that the path $P_3$ exists is $\frac{1}{4} \frac{1}{7} \frac{1}{9} \frac{1}{11} \frac{1}{14} \frac{1}{16}$, which is greater than that of $P_2$.

\begin{property}\label{border vertices}
Holding all other factors constant, a directed path with more vertices on the border of the GIG digraph is more likely to occur in a randomly labeled GIG digraph than a directed path with fewer vertices on the borders. 
\end{property}

If a vertex is on the border or one of the four corners of the GIG digraph, then it  has only three or two neighbors, respectively. Similar to the argument in Property \ref{more turns higher prob}, there will be a decrease of $1$ or $2$ in the value of the corresponding $K_i$ value, which leads to a decrease of each $K_j$ for $j>i$. Since the probability of the existence of the path is the product of the reciprocals of all of the $K_j$ value, such a probability will increase when some $K_i$ decreases.

Based on the above properties, we will be able to find a lower bound for the probability that a given path in a GIG digraph exists. Suppose $V_{i,j}, V_{i',j'}$ are two vertices in a GIG digraph. There are $\binom{|i'-i|+|j'-j|}{|i'-i|}$ possible paths of length $|j'-j|+|i'-i|$ that connect the two vertices. Note that $\binom{|i'-i|+|j'-j|}{|i'-i|}$ is only the number of shortest paths that connect the two vertices. There might exist numerous longer paths that also connect these two selected vertices. However, these cases will be much rarer. Two more steps in a path will dramatically bring down its probability of occurrence. The longer the path, the more significant the decrease is.

Let path $P_l$ denote a directed path of length $|j'-j|+|i'-i|$ from vertex $V_{i,j}$ to $V_{i',j'}$ that has the least probability of existence among all paths with this same shortest length. Denote by $L$ the probability that path $P_l$ occurs in a randomly labeled GIG digraph. We claim that among all the ways to put labels on vertices, the probability that two vertices $V_{i,j}$ and $V_{i',j'}$ are connected with at least one path is greater than or equal to $\binom{|i'-i|+|j'-j|}{|i'-i|} \times L$. Since the probability that the two vertices are connected by a shortest path is greater than or equal to this lower bound, and there might exist some longer paths that connect these two vertices, this bound is indeed a valid lower bound. The path $P_l$ is chosen based on the characteristics of potential shortest paths that connect vertices $V_{i,j}$ and $V_{i',j'}$. The path $P_l$ will possess as many properties discussed above as possible. For instance, it has the minimum number of turns and has turns in later part of the path; vertices of this path will have as many neighbors as possible.

\section{Probabilities and statistical properties of sinks}
In this section, we first look at the probability that several selected vertices in a randomly labeled GIG digraph are sinks. 
We then derive the mean and variance for the number of sinks.
%
%
%
%
%
\subsection{Probability of multiple sinks}
\begin{theorem}\label{multiple sinks}
Let $V_1, \ldots, V_i$ denote $i$ random vertices in an $m \times n$ GIG digraph where their labels follow $v_1 < \cdots < v_i$. The probability that $V_1, \ldots, V_i$ are sinks is $\frac{1}{K_1} \cdot  \frac{1}{K_2} \cdot \cdots \cdot \frac{1}{K_{i}}$.
\end{theorem}

\begin{proof}
We will follow a similar proof to that of Theorem \ref{N step probability}. Since $V_1, \dots, V_i$ are sinks and $v_1 < \cdots < v_i$, sink $V_{i}$ has the largest label among vertices in set $N(V_1, \ldots, V_i)$. Any other values are valid ones to label the other $K_i - 1$ vertices in set $N(V_1, \ldots, V_i)$. Thus, there are $\binom{K_i-1}{K_{i-1}}$ ways to choose labels for vertices in set $N(V_1, \ldots, V_{i-1})$. Fixing the labels of these $K_{i-1}$ vertices and the label of vertex $V_i$, there will be $(K_i- K_{i-1} - 1)!$ ways to assign labels to other neighbors of $V_{i}$ that are not in the set $N(V_1, \ldots, V_{i-1})$.  

For $1 \leq l \leq i$, let set $N(V_1, \ldots, V_{i-l})$ be the set that contains the $i-l$ sinks with the smallest $i-l$ labels and the neighbors of these sinks. Applying the argument provided above inductively to the set $N(V_1, \ldots, V_{i-l})$ for all possible values of $l$, we will have the following analytical formula to enumerate the number of possible labellings such that $V_1, \ldots, V_i$:
\begin{equation*}
\binom{K_i-1}{K_{i-1}} \binom{K_{i-1}-1}{K_{i-2}} \cdots \binom{K_2-1}{K_1} (K_i- K_{i-1} - 1)! (K_{i-1}- K_{i-2} - 1)! \cdots (K_1 - 1)!
\end{equation*}

The number of ways to put the $K_i$ labels on these $K_i$ vertices is $K_i!$. The probability that $V_1, \ldots, V_i$ in a randomly labeled GIG digraph are sinks is thus the formula above over  $K_i!$, which can be easily simplified into $\frac{1}{K_1} \cdot \frac{1}{K_2} \cdot \cdots \cdot  \frac{1}{K_{i}}$.
\end{proof}

\begin{figure}[ht]
\centering
\begin{tikzpicture}
[scale=0.8,every node/.style={draw=black,circle}]

    \node[scale = 1.7] (1) at (1.5,0) {};
    \node[scale = 1.7, fill] (8) at (3,0) {};
    \node[scale = 1.7] (10) at (4.5,0) {};
    \node[scale = 1.7] (4) at (0,1.5) {};
    \node[scale = 1.7, fill] (7) at (1.5,1.5) {};
    \node[scale = 1.7] (3) at (3,1.5) {};
  
    \node[scale = 1.7] (9) at (1.5,3) {};
    
\end{tikzpicture}
\caption{Two Sinks and Their Neighbors} \label{fig:Sinks}
\end{figure}
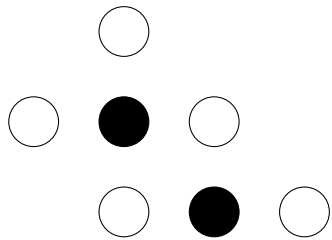

\begin{remark}\normalfont
With Theorem \ref{multiple sinks}, the probability that $i$ vertices are sinks can be calculated by determining $K_1, K_2, \ldots, K_i$ for each of the possible descending orders of these $i$ vertices and summing up the $i!$ possible descending orders of these $i$ vertices. For example, Figure \ref{fig:Sinks} is a reduced GIG digraph that depicts all neighbors of the two black vertices. We would like to know the probability that the two black vertices are sinks.

Denote the black vertex on the upper left by $V_a$ and the black vertex on the lower right by $V_b$. Note that $|N(V_a)| = 5$, $|N(V_b)| = 4$, and $|N(V_a, V_b)| = 7$. When $v_a > v_b$, the probability that a both black vertices are sinks is $\frac{1}{5} \cdot \frac{1}{7}$; when $v_b > v_a$, such a probability is $\frac{1}{4} \cdot \frac{1}{7}$. Thus, the probability that both $V_a$ and $V_b$ are sinks is $\frac{1}{35} + \frac{1}{28}$. Notice that the descending order of the labels of $V_a$ and $V_b$ matters when determining the probability that these vertices are sinks. We cannot simply calculate the probability of one particular descending order of labels and multiply it by the number of possible descending orders.  

\end{remark}

%
%
%
%
\begin{proposition} \label{two sinks independent}
Vertex $V_a$ is a sink and vertex $V_b$ is a sink are independent events if and only if the Euclidean distance between $V_a$ and $V_b$ is more than 2.
\end{proposition}

\begin{proof}
To simplify notation, we denote the number of vertices in the neighborhood of a single vertex $V_a$, $|N(V_a)|$, by $k_a$. If the Euclidean distance between $V_a$ and $V_b$ is more than $2$, then they do not share any neighbors. By Theorem \ref{multiple sinks}, the probability that $V_a$ is a sink is $\frac{1}{k_a}$, the probability that $V_b$ is a sink is $\frac{1}{k_b}$. Since $V_a$ and vertex $V_b$ do not share any neighbors, the probability that both $V_a$ and $V_b$ are sinks is $ \frac{1}{k_a} \cdot \frac{1}{k_a + k_b} +  \frac{1}{k_b} \cdot \frac{1}{k_a + k_b}$. Note that
\[
\left( \frac{1}{k_a} + \frac{1}{k_b} \right) \frac{1}{k_a + k_b}  = \frac{k_a + k_b}{k_a k_b}  \frac{1}{k_a + k_b} = \frac{1}{k_a} \frac{1}{k_b}.
\]

This completes the proof. 
\end{proof}

\subsection{Mean and variance of the number of sinks}
%
%
%
%
\begin{theorem}
The expected value of the number of sinks in an $m \times n$ GIG digraph is $\frac{mn}{5} + \frac{m+n}{10} + \frac{2}{15}$. This formula applies to cases where $m \geq 3$ and $n \geq 3$.
\end{theorem}

\begin{proof}
Let $X$ denote the number of sinks in an $m \times n$ GIG digraph, and let $X_i$ be the number of sinks at vertex $V_i$. Note that $X_i$ equals to either $0$ or $1$. The four vertices in the four corners have $2$ neighbors. Vertices in the first and last row, first and last column have $3$ neighbors. The remaining non-border vertices each have $4$ neighbors. By Theorem \ref{multiple sinks}, the probability that each of the above types of vertices is a sink is $\frac{1}{3}$, $\frac{1}{4}$ and $\frac{1}{5}$ respectively. Note that
\begin{equation}\label{ex px}
E(X_i) = 0 \times P(X_i = 0) + 1 \times P(X_i = 1) = P(X_i),
\end{equation}
\begin{equation}\label{ex sum}
E(X) = E(\displaystyle \sum_{i=1}^{mn} X_i) = \sum_{i=1}^{mn} E(X_i) = \sum_{i=1}^{mn} P(X_i).
\end{equation}

Therefore, the expected value for the number of sinks in an $m \times n$ GIG digraph is the sum of each vertex's probability to be a sink. The resulting expected value is 
\[ \frac{(m-2)(n-2)}{5} + \frac{2(m-2) + 2(n-2)}{4} + \frac{4}{3} = \frac{mn}{5} + \frac{m+n}{10} + \frac{2}{15}\].
\end{proof}
%
%
%
%
\begin{theorem}
The variance of the number of sinks in an $m \times n$ GIG digraph is $\frac{13mn}{225} + \frac{m+n}{150} + \frac{52}{1575}$. This formula applies to cases where $m \geq 6$ and $n \geq 6$.
\end{theorem}

\begin{proof}
Let $X$ denote the number of sinks in an $m \times n$ GIG digraph, and let $X_i$, $X_j$ denote the number of sinks at vertices $V_i$ and $V_j$. Note that the variance of the sum of indicator random variables is
\begin{equation*}
Var(X) = \sum_{i=1}^{mn} Var(X_i) + \sum_{j=1}^{mn}  \sum_{i \neq j}^{mn} Cov(X_i, X_j).
\end{equation*}
Since each individual $X_i$ follows a Bernoulli distribution, $Var(X_i) = p(1-p)$, where $p = P(X_i = 1)$. Thus, $Var(X_i)$ equals $\frac{2}{9}$, $\frac{3}{16}$ or $\frac{4}{25}$ if the vertex has $2$, $3$ or $4$ neighbors, respectively. Note that the $4$ vertices in the corner have $2$ neighbors, the $2 \times (m-2+n-2)$ vertices in the first and last row, and first and last column have $3$ neighbors and the remaining $(m-2)(n-2)$ vertices have $4$ neighbors.

Next we will look at the covariances between each pair of vertices in the grid. Note that
\begin{equation*}
Cov(X_i, X_j) = E(X_i X_j) - E(X_i)E(X_j)
\end{equation*}
\begin{equation*}
E(X_i X_j) = \sum_{i = 0}^{1} \sum_{j = 0}^{1} X_i X_j f(X_i,X_j) =P(X_i=1,X_j=1).
\end{equation*}
In other words, $E(X_i X_j)$ is the probability that both $X_i$ and $X_j$ are sinks. Note that the value of $E(X_i X_j)$ varies across four cases. If $V_i$ and $V_j$ do not share any neighbors, then $V_i$ is a sink and $V_j$ is a sink are independent events, which implies a zero covariance. By remark \ref{multiple sinks}, $E(X_i X_j)$ in the other three cases equals
\begin{enumerate}[label=(\roman*)]
\item $\frac{1}{k_i+k_j-1} \left( \frac{1}{k_i} + \frac{1}{k_j} \right)$ if $V_i$ and $V_j$ share $1$ neighbor. 
\item $\frac{1}{k_i+k_j-2} \left( \frac{1}{k_i} + \frac{1}{k_j} \right)$ if $V_i$ and $V_j$ share $2$ neighbors.
\item $0$ if $V_i$ and $V_j$ are neighbors.
\end{enumerate}
Thus, the covariance of the above three cases will be $\frac{1}{(k_i+k_j-1)k_i k_j}$, $\frac{2}{(k_i+k_j-2) k_i k_j}$ and $-\frac{1}{k_i k_j}$ respectively.

For each vertex in the GIG digraph, we add its variance and its covariance with all other vertices on the grid. Then we sum these values for all vertices on the grid. The result is $\frac{13mn}{225} + \frac{m+n}{150} + \frac{52}{1575}$.
\end{proof}

\section{Expected Component Size}
%
%
%
%
In this section, we consider partitioning the GIG digraph into components, where each component consists of a sink and the set of vertices with a directed path to that sink. For example, Figure \ref{fig:GIG_two} shows an unlabeled $2 \times 3$ GIG digraph with two sinks:  $V_{1,2}$ and $V_{2,3}$. One component consists of $V_{1,2}$ and the four vertices with directed paths to it. The other component is $V_{2,3}$ alone. We are interested in the expected size of a component in a randomly labeled GIG digraph. Properties of the components will add to our understanding of the GIG digraph and the potential for applications to global optimization algorithms.

\begin{figure}[ht]
\centering
\begin{tikzpicture}
[scale=0.8,every node/.style={draw=black,circle}]

    \node[scale = 1.7] (4) at (0,0){};
    \node[scale = 1.7] (7) at (1.5,0) {};
    \node[scale = 1.7] (3) at (3,0) {};
    \node[scale = 1.7] (2) at (0,1.5) {};
    \node[scale = 1.7] (9) at (1.5,1.5) {};
    \node[scale = 1.7] (5) at (3,1.5) {};
    
    \draw[->] (2) to (9);
    \draw[->] (5) to (9);
    \draw[->] (4) to (7);
    \draw[->] (7) to (9);

\end{tikzpicture}
\caption{An unlabeled $2 \times 3$ GIG digraph with two components} \label{fig:GIG_two}
\end{figure}
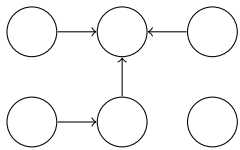

\begin{theorem}\label{convergence}
In an $m \times n$ GIG digraph, let $M = \max(m,n)$. The expected size of each component in such a GIG digraph is bounded above by 
\[
\sum_{n = 1}^{M} 4n \sum_{l = n}^{M^2} \binom{l}{\lceil \frac{l}{2} \rceil} \prod_{i = 1}^{l} \frac{1}{2+i}.\]
\end{theorem}

\begin{proof}

Let $X$ denote the number of the vertices in a component with a sink $V_s$, and let $X = \sum_{p=1}^{mn} X_p$, where $X_p = 1$ if vertex $V_p$ is located in this component, and $X_p = 0$ otherwise. By Equation (\ref{ex px}) and Equation (\ref{ex sum}), $E(X)$ can be calculated by summing up the probabilities for each vertex in the GIG digraph to reach the sink $V_s$. Because there can be at most one path from $V_i$ to $V_s$, the probability that vertex $V_i$ reaches the sink $V_s$ is less than or equal to the sum over all possible paths between $V_i$ and $V_s$ of the probability of each path. We will calculate the sum of the probability for all vertices in the  GIG digraph to reach the sink $V_s$ along a path. For simplicity in expressing an upper bound, we pad the GIG digraph with additional rows or columns, so it is a square with $M$ rows and columns. 

Note that there are less than or equal to $4n$ vertices at distance $n$ from the sink, and the maximum distance of a vertex from the sink is $M$. Further, no path in the GIG digraph can be longer than $M^2$, and the number of paths of length $l$ is maximized when taking $\lceil \frac{l}{2} \rceil$ vertical steps and  $l - \lceil \frac{l}{2} \rceil$ horizontal steps. Therefore, enumerating the vertices by their distance from the sink, and enumerating the paths by their length, we have
\[ E(X) < \sum_{n=1}^{M} 4n \sum_{l=n}^{M^2} \binom{l}{\lceil \frac{l}{2} \rceil } \textrm{P(a path of length } l). \]

By Theorem \ref{N step probability}, the probability of a path of length $l$ is bounded above by $\prod_{i = 1}^{l} \frac{1}{2+i}$, since each vertex in the path adds at least one new neighbor to the set of neighbors of vertices in the path. Therefore, 
\[ E(X) < \sum_{n=1}^{M} 4n \sum_{l=n}^{M^2} \binom{l}{\lceil \frac{l}{2} \rceil } \prod_{i = 1}^{l} \frac{1}{2+i}. \]

\end{proof}

%
%
%

\begin{theorem}
The expected size of the maximum component in an $m \times n$ GIG digraph converges as the size of the GIG digraph goes to infinity. 
\end{theorem}

\begin{proof}

We need to show that 
\[
\sum_{n = 1}^{\infty} 4n \sum_{l = n}^{\infty} \binom{l}{\lceil \frac{l}{2} \rceil} \prod_{i = 1}^{l} \frac{1}{2+i} \]
converges. First, we will show that $b_n = \sum_{l = n}^{\infty} \binom{l}{\lceil \frac{l}{2} \rceil} \prod_{i = 1}^{l} \frac{1}{2+i}$ converges for all $n \geq 1$, and that $\frac{b_{n+1}}{b_n} < \frac{2}{3}$ for $n \geq 2$.

The ratio between two consecutive terms in the series $b_n$ is
\begin{align*}
\frac{\binom{l+1}{\lceil \frac{l}{2} \rceil} \prod_{i = 1}^{l+1} \frac{1}{2+i}}
{\binom{l}{\lceil \frac{l}{2} \rceil} \prod_{i = 1}^{l} \frac{1}{2+i}}
= 
\frac{l+1}{\lceil \frac{l+1}{2} \rceil} \frac{1}{3+l}
\leq 
\frac{l+1}{\left( \frac{l+1}{2} \right)} \frac{1}{3+l}
=
\frac{2}{3+l}
<
\frac{2}{l+1}.
\end{align*}
By the ratio test, $b_n$ converges for all $n \geq 1$. 

For fixed $n \geq 2$, let  $A = \binom{n}{\lceil \frac{n}{2} \rceil} \prod_{i = 1}^{n} \frac{1}{2+i} $. Then 
\begin{align*}
b_{n+1} &=  \sum_{l = n+1}^{\infty} \binom{l}{\lceil \frac{l}{2} \rceil} \prod_{i = 1}^{l} \frac{1}{2+i} \\
&<  \frac{2}{n+1} A + \frac{2}{n+1} \frac{2}{n+2}A + \cdots \\
&<  A \sum_{i=1}^{\infty} \left ( \frac{2}{n+1} \right) ^i \\
&= \frac{2A}{n-1} \\
&\leq 2A.
\end{align*}

Then 
\[
\frac{b_{n+1}}{b_n} = \frac{b_{n+1}}{A + b_{n+1}} \leq
\frac{2A}{A + 2A} = \frac{2}{3},
\]

and
\begin{align*}
\sum_{n = 1}^{\infty} 4n b_n &= 4b_1 + \sum_{n=2}^\infty 4n b_n \\
&\leq 4b_1 + \sum_{n=2}^\infty 4n \left( \frac{2}{3} \right)^{n-2} b_2 \\
&= 4b_1 + 48b_2  \\
&< \infty
\end{align*}

Therefore, the expected size of the maximum component in a GIG digraph converges. 

\end{proof}

%
%
%
%
\bibliographystyle{amsplain}

\begin{thebibliography}{10}

\bibitem {main} R. Allen, L. Heyer, R.I. Nishat, S. Whitesides, Grid Proximity Graphs: LOGs, GIGs and GIRLs, in \textit{CCCG}. 2013.

\bibitem {enum paper} J. Chester, L. Edlin, J. Galeota-Sprung, B. Isom, A. Moore, V. Perkins, A. Campbell, T. Eckdahl, L. Heyer, J. Poet, On counting limited outdegree grid digraphs and greatest increase grid digraphs, \textit{Involve, a Journal of Mathematics}, 9 (2): 211–221, 2016.

\bibitem {artificial intelligence} S. Russell and P. Norvig, \textit{Artificial Intelligence: A Modern Approach, Third Edition}, Cambridge University Press, Upper Saddle River, N.J., U.S.A., 2010.

\bibitem {lourenco} Lourenço, Helena R., Olivier C. Martin, and Thomas Stützle, Iterated local search, in \textit{Handbook of metaheuristics}, Springer, Boston, MA, 320-353, 2003.


\end{thebibliography}

\end{document}